\documentclass[12pt]{amsart}

\usepackage[OT2, T1]{fontenc}

\usepackage{amscd}
\usepackage{amsmath}
\usepackage{amssymb}
\usepackage{mathrsfs} 			% for \mathscr (script letters)
\usepackage{units}
\usepackage[all]{xy}

\usepackage{algorithm}
\usepackage{algorithmic}

\def\frk{\frak}               % font for "Fraktur"

\def\pp{{\frk p}}

\def\mm{{\frk m}}

\def\Phi{{\frk n}}
\def\Phi{{\frk N}}
\def\opn#1#2{\def#1{\operatorname{#2}}} % to make operators
\opn\chara{char} \opn\length{\ell} \opn\pd{pd} \opn\rk{rk}
\opn\projdim{proj\,dim} \opn\injdim{inj\,dim} \opn\rank{rank}
\opn\depth{depth} \opn\sdepth{sdepth} \opn\fdepth{fdepth}
\opn\grade{grade} \opn\height{height} \opn\embdim{emb\,dim}
\opn\codim{codim}  \opn\min{min} \opn\max{max}

\opn\Tr{Tr} \opn\bigrank{big\,rank}
\opn\superheight{superheight}\opn\lcm{lcm}
\opn\trdeg{tr\,deg}%\emph{
\opn\reg{reg} \opn\lreg{lreg} \opn\ini{in} \opn\lpd{lpd}
\opn\size{size}
\opn\div{div} \opn\Div{Div} \opn\cl{cl} \opn\Cl{Cl}
\opn\Spec{Spec} \opn\Supp{Supp} \opn\supp{supp} \opn\Sing{Sing}
\opn\Ass{Ass} \opn\Min{Min}
\opn\Ann{Ann} \opn\Rad{Rad} \opn\Soc{Soc}
\opn\Im{Im} \opn\Ker{Ker} \opn\Coker{Coker} \opn\Am{Am}
\opn\Hom{Hom} \opn\Tor{Tor} \opn\Ext{Ext} \opn\End{End}
\opn\Aut{Aut} \opn\id{id}  \opn\deg{deg}

\opn\nat{nat}
\opn\pff{pf}%   \pf exists already
\opn\Pf{Pf} \opn\GL{GL} \opn\SL{SL} \opn\mod{mod} \opn\ord{ord}
\opn\Gin{Gin} \opn\Hilb{Hilb}
\opn\aff{aff} \opn\con{conv} \opn\relint{relint} \opn\st{st}
\opn\lk{lk} \opn\cn{cn} \opn\core{core} \opn\vol{vol}
\opn\link{link} \opn\star{star}
\opn\gr{gr}

\def\pot#1#2{#1[\kern-0.28ex[#2]\kern-0.28ex]}

\opn\dirlim{\underrightarrow{\lim}}
\opn\inivlim{\underleftarrow{\lim}}
\let\iso=\cong

\let\to=\rightarrow

\def\Implies{\ifmmode\Longrightarrow \else
        \unskip${}\Longrightarrow{}$\ignorespaces\fi}
\def\implies{\ifmmode\Rightarrow \else
        \unskip${}\Rightarrow{}$\ignorespaces\fi}
\def\iff{\ifmmode\Longleftrightarrow \else
        \unskip${}\Longleftrightarrow{}$\ignorespaces\fi}

\let\:=\colon
\newtheorem{Theorem}{Theorem}[]
\newtheorem{Lemma}[Theorem]{Lemma}
\newtheorem{Corollary}[Theorem]{Corollary}
\newtheorem{Proposition}[Theorem]{Proposition}

\theoremstyle{definition}
\newtheorem{Example}[Theorem]{Example}

\newtheorem{Remark}[Theorem]{Remark}

\newtheoremstyle{subsection-tweak}
   {11pt}
   {3pt}%
   {}
   {}%
   {\bfseries}
   {}%
   {.5em}
   {\thmnumber{\@{#1}{}\@{#2}.}%
    \thmnote{~{\bfseries#3.}}}    

\newcounter{numberingbase}

\theoremstyle{subsection-tweak}
\newtheorem{bpp}[Theorem]{}
\newtheorem{bppt}[numberingbase]{}
\newcommand{\bbpp}{\begin{bpp}}
\newcommand{\eepp}{\end{bpp}}
\newcommand{\bbppt}{\begin{bppt}}
\newcommand{\eeppt}{\end{bppt}}

\theoremstyle{theorem}

\theoremstyle{definition}

\newcommand{\val}{\mathrm{val}}		% Valuation

\let\epsilon\varepsilon
\let\phi=\varphi
\def\qed{\ifhmode\textqed\fi
      \ifmmode\ifinner\quad\qedsymbol\else\dispqed\fi\fi}
\def\textqed{\unskip\nobreak\penalty50
       \hskip2em\hbox{}\nobreak\hfil\qedsymbol
       \parfillskip=0pt \finalhyphendemerits=0}
\def\dispqed{\rlap{\qquad\qedsymbol}}

\opn\dis{dis}
\def\pnt{{\raise0.5mm\hbox{\large\bf.}}}

\opn\Lex{Lex}

\begin{document}

\title{Filtered colimits of complete intersection  algebras.}

\author{ Dorin Popescu}

\address{Simion Stoilow Institute of Mathematics of the Romanian Academy,
Research unit 5, P.O. Box 1-764, Bucharest 014700, Romania,}

\address{University of Bucharest, Faculty of Mathematics and Computer Science
Str. Academiei 14, Bucharest 1, RO-010014, Romania,}

\address{ Email: {\sf dorin.m.popescu@gmail.com}}

\dedicatory{In the memory  of Lucian B\u adescu.}

\begin{abstract} This is mainly a small exposition on extensions of valuation rings. 
\end{abstract}

\maketitle
\section{Introduction}

After Zariski \cite{Z} a valuation ring V containing a field $K$ of characteristic zero is a filtered union of its smooth $K$-subalgebras. When $V$ contains a field $K$ of positive characteristic this result may fail, as we believe. A possible support for this 
idea is given by Example \ref{e}.

Then we ask if $V$ is a filtered union of its complete intersection $K$-subalgebras essentially of finite type. This could hold as our Theorem \ref{T'} hints.
 We remind that for a given complete intersection local ring $(R,\mm,k)$ a minimal free resolution over $R$ of the residue field $k$ is described in \cite[Theorem 4]{Ta} and others as 
 \cite[Theorem 2.7]{As}, \cite[§6]{Av}, \cite[Proposition 1.5.4]{GL},  \cite[§3]{HM}, \cite[ Theorem 2.5]{AHS}, \cite[Theorem 4.1]{NV}. As Oana Veliche asked, we wonder if  a minimal free resolution of the residue field of $V$ could be described using somehow our Theorem \ref{T'}.

We owe thanks to Oana Veliche and the  referee who hinted some misprints in our paper.

\section{Henselian rings and smooth algebras}

A {\em Henselian} local ring is a local ring 
 $(A,\mm)$ for which the Implicit Function Theorem hods, that is 
 for every system of polynomials  $f=(f_1,\ldots ,f_r) $ of $A[Y]$,
$Y=(Y_1,\ldots,Y_N)$, $N\geq r$ over  $A$ and every solution  $y=(y_1,\ldots,y_N)\in A^N$     modulo $\mm$ of $f$ in   $A$ such that an $r\times r$-minor $M$ of Jacobian matrix $(\partial f_i/\partial Y_j)_{1\leq i\leq r,1\leq j\leq N}$  satisfies $M(y)\not \in \mm$ there exists a solution $\tilde y$ of $f=0$ in  $A$ with $y\equiv {\tilde y}$ modulo $\mm$.

An  $A$-algebra $B$ is  {\em
smooth} if $B$ is a localization of an $A$-algebra of type  $(A[Y]/(f))_M$, where  $f=(f_1,\ldots ,f_r) $  is a system of polynomials
in  $Y=(Y_1,\ldots,Y_N)$, $N\geq r$ over  $A$ and  $M$ is an $r\times r$-minor of $(\partial f/\partial Y)$. Thus  $(A,\mm)$ is Henselian if for any smooth $A$-algebra $B$  any $A$-morphism $B\to A/\mm$ can be lifted to an $A$-morphism $w:B\to A$. Thus  $w$ is a retraction of the canonical morphism $u:A\to B$, that is $wu=1_A$. A retraction of $A\to B$ maps the solutions of a system of polynomials $g=(g_1,\ldots,g_s)\in A[Y]^s$ in $B$ in solutions of $g$ in $A$.

We remind that a  filtered colimit  is a limit indexed by a small category that is filtered
  (see \cite[002V]{SP} or \cite[04AX]{SP}). A filtered  union is a filtered direct limit in which all objects are subobjects of the final colimit,
   so that in particular all the transition arrows are monomorphisms.

\begin{Proposition} \label{p} Let $(A,\mm)\to (A',\mm')$ be a morhism of local rings with $A/\mm\iso A'/\mm'$, which is a filtered colimit of smooth $A$-algebras and 
 $g\in A[Y]^s$  a system of polynomials which has  solutions in $A'$. If $(A,\mm)$ is Henselian, then $g$ has also solutions in $A$.
 \end{Proposition}

\begin{proof}
    Note that a solution $y$ in $A'$ of a system of polynomials $g$ over $A$ comes from a solution $y'$ of $g$ in a smooth $A$-algebra $B$, which is mapped in a solution of $g$ in $A$ by a retraction of $A\to B$.
\end{proof}   
 
  Examples of morphisms  $u:(A,\mm)\to (A',\mm')$, which are filtered colimits of smooth $A$-algebras
are given by the following theorems.  

The following result  is the General Neron Desingularization \cite{Po0}, \cite{Po0'}, \cite{S}.  

 \begin{Theorem} \label{T0}   Let  $u:(A,\mm)\to (A',\mm')$ be a flat morphism of Noetherian local rings. The following statements are equivalent:
 \begin{enumerate}
 \item $u$ is  a filtered colimit of smooth $A$-algebras.
 
 \item $u$ is regular, that is for every prime ideal $\pp\in \Spec A$ and every finite field extension
$K$ of the fraction field of $A/\pp$ the ring $K\otimes_{A/\pp} A'$ is regular.
 \end{enumerate}
 \end{Theorem}
 
Let $(A,\mm)$ be a Noetherian local ring  and $ {\hat A}$ its completion in the $\mm$-adic topology. $A$ is {\em excellent} if the completion map $A\to {\hat A}$ is regular. The following corollary holds by Proposition \ref{p}. 
  
  \begin{Corollary} (\cite{Po0}) Let  $(A,\mm)$ be an excellent  Henselian local ring  and $ {\hat A}$ its completion. Then every system of polynomials over $A$ which has a solution in ${\hat A} $ has also one in $A$. In particular $A$ has the property of Artin approximation; this was conjectured by M. Artin in \cite{A}.
  \end{Corollary}

\begin{Theorem} \label{T1'}  (Zariski \cite{Z}) A valuation ring containing a field $K$ of characteristic zero is a    a filtered union of its smooth $K$-subalgebras.
\end{Theorem}  

In the same frame, the following result holds.     
     
\begin{Theorem}(\cite[Theorem 2]{P1}) \label{T1} Let $V\subset V'$ be an extension of valuation rings containing $\bf Q$,            $K\subset K'$ its fraction field extension, $\Gamma\subset \Gamma'$ the value group extension of $V\subset V'$  and $\val: K'^{*}\to \Gamma'$ the valuation of $V'$. Then $V'$  is a filtered  colimit of smooth $V$-algebras if and only if  the following statements hold:
\begin{enumerate}

\item For each $q\in \Spec V$ the ideal $qV'$ is prime,

\item  For any prime ideals $q_1,q_2\in \Spec V$ such that $q_1\subset q_2$ and height$(q_2/q_1)=1$  and any $x'\in q_2V'\setminus q_1'$ there exists $x\in V$ such that $\val(x')=\val(x)$,  where $q_1'\in \Spec V'$ is the prime ideal corresponding to the maximal ideal of $V_{q_1}\otimes_V V'$, that is the maximal prime ideal of $V'$ lying on $q_1$.
\end{enumerate}
\end{Theorem}

An {\it{immediate extension}} of valuation rings is an extension inducing trivial extensions on residue fields and group value extensions. As a  consequence  of Theorem \ref{T1} we can prove the following:

\begin{Corollary}(\cite[Theorem 21]{P})\label{c0} If $V\subset V'$ is an immediate extension of valuation rings   containing $\bf Q$
then $V'$ is a filtered colimit of smooth $V$-algebras. 
\end{Corollary}
 
 The following corollary holds by Proposition \ref{p}. 
\begin{Corollary}(\cite[Proposition 18]{P1})\label{C0} Let $V\subset V'$ be an extension of valuation rings containing $\bf Q$ with the same residue field, $K\subset K'$ its fraction field extension, $\Gamma\subset \Gamma'$ the value group extension of $V\subset V'$  and $\val: K'^{*}\to \Gamma'$ the valuation of $V'$. Assume that $V$ is Henselian and the statements (1), (2) of Theorem \ref{T1} hold (for example when $V\subset V'$ is immediate).  Then every system of polynomials over $V$ which has a solution in $V' $ has also one in $V$.
\end{Corollary}

Let $V$ be a valuation ring, $\lambda$ be a fixed limit ordinal  and $v=\{v_i \}_{i < \lambda}$ a sequence of elements in $V$ indexed by the ordinals $i$ less than  $\lambda$. Then $v$ is called \emph{pseudo convergent} if 

$\val(v_{i} - v_{i''} ) < \val(v_{i'} - v_{i''} )     \ \ \mbox{for} \ \ i < i' < i'' < \lambda$
(see \cite{Kap}, \cite{Sch}).
A  \emph{pseudo limit} of $v$  is an element $w \in V$ with 

$ \val(w - v_{i}) < \val(w - v_{i'}) \ \ \mbox{(that is,} \ \ \val(w -  v_{i}) = \val(v_{i} - v_{i'}) \ \ \mbox{for} \ \ i < i' < \lambda$.

The following example shows that Corollary \ref{c0} fails in positive characteristic.
 
 \begin{Example} (\cite[Example 3.1.3]{Po1},\cite{O}, \cite{P-1}) \label{e}. Let $k$  be a field of characteristic $p>0$, $X$ a variable, $\Gamma={\bf Q}$  and $K$ the fraction field of the group algebra $k[\Gamma]$, that is the rational functions in $(X_q)_{q\in {\bf Q}}$. 
 Let $P$ be the field of all formal sums $z= \sum_{n\in {\bf N}} c_nX^{\alpha_n}$, 
 where $(\alpha_n)_{n\in {\bf N}}$ is a monotonically increasing sequence  from $\Gamma$ and $c_n\in k$. 
 Set $\val(z)=\alpha_s$, where $s=\min\{n\in {\bf N}: c_n\not=0\}$ if $z\not =0$
  and let $V$ be the valuation ring defined by $\val:P^*\to \Gamma$, $z\to \val(z)$. 
   
   Let $\rho_n=(p^{n+1}-1)/(p-1)p^{n+1}$, $y=-1+\sum_{n\geq 0} (-1)^n X^{\rho_n}$ and
    $a_i=-1+\sum_{0\leq n\leq i} (-1)^n X^{\rho_n}$,   
  We have $1+\rho_n=p(\rho_{n+1})$
   for $n\geq 0$ and 
    $p\rho_0=1$ and 
   $y$ is a pseudo limit of the pseudo convergent sequence
    $a=(a_i)_{i\in {\bf N}}$, which has no pseudolimit in $K$. Then $y$ is a root of the separable polynomial
     $g=Y^p+XY+1\in K[Y]$ and the algebraic separable extension $V_0=V \cap K\subset V_1=V \cap K(y)$ is not dense,
     in particular $V_1$ is not a filtered colimit of smooth $V_0$-algebras
      (apply \cite[Theorem 6.9]{Po1},
      or \cite[Theorem 2]{P-1}). 
\end{Example}

However, there exist an important result when char $k=p>0$. 

\begin{Theorem} \label{T2}
 (B. Antieau, R. Datta \cite[Theorem 4.1.1]{AD}) Every perfect valuation ring of characteristic $p>0$ is a filtered union of its smooth ${\bf F}_p$-subalgebras. 
 \end{Theorem}
We mention also the following result.
\begin{Theorem}(\cite{P0}) \label{ultra} Let V be a one dimensional valuation ring with value group $\Gamma$ containing its residue field
and $U$ a set with
 card$(U)>$card$(\Gamma)$. Then there exists an ultrafilter $\mathcal U$ on $U$ such that taking the ultrapower $\tilde V$ of $V$ with respect to $\mathcal U$ and ${\bar V}=
 {\tilde V} / \cap_{z\in V,z\not =0} z {\tilde V }$ the following assertions hold:

\begin{enumerate} 
\item $\bar V$ contains the completion of $V$.
\item If $V'$ is an immediate extension of $V$ contained in $\bar V$ with $V'/V$ separable then $V'$ is a filtered  colimit of  smooth $V$-algebras.
\end{enumerate}
\end{Theorem}

\section{Complete intersection algebras}

For a general immediate extension of valuatin rings $V\subset V'$ we should expect that $V'$ is a  filtered  union of its complete intersection $V$-subalgebras. In the Noetherian case a morphism of rings is a filtered direct limit of smooth algebras iff it is a regular morphism (see Theorem \ref{T0}).

A {\em complete intersection} $V$-algebra {\it{essentially of finite type}} is a local $V$-algebra of type $C/(P)$, where $C$ is a localization of a polynomial $V$-algebra of finite type and $P$ is a regular sequence of elements of $C$.  Theorem \ref{T} stated below says that $V'$ is a  filtered
 union of its $V$-subalgebras of type $C/(P)$. Since $V'$ is local it is enough to say that $V'$ is  a  filtered
 union of its $V$-subalgebras of type $T_h/(P)$,  $T$ being a polynomial $V$-algebra of finite type, $0\not =h\in T$ and $P$ is a regular sequence of elements of $T$. Clearly, $ T_h$ is a smooth $V$-algebra and in fact it is enough to say that $V'$ is  a  filtered
 union of its $V$-subalgebras of type $G/(P)$, where $G$  is a smooth $V$-algebra of finite type and $P$ is a regular sequence of elements of $G$. Conversely, a $V$-algebra of such type $G/(P)$ has the form $T_h/(P)$ for some $T,h,P$ using \cite[Theorem 2.5]{S}. By abuse we understand by a {\em complete intersection } $V$-algebra of finite type a $V$-algebra of such type $G/(P)$, or $T_h/(P)$ which are not assumed to be flat over $V$.

 The following  lemma is a variant  of  Ostrowski (\cite[ page 371, IV and III]{O}, see also \cite[(II,4), Lemma 8]{Sch}) and \cite[Lemma 2]{P2}.

\begin{Lemma}(Ostrowski) \label{o1} Let $\beta_1,\ldots,\beta_m$ be any elements of an ordered abelian group $G$, $\lambda$ a limit ordinal and
 let $\{\gamma_s\}_{s<\lambda}$ be an increasing sequence of elements of G. Let $ t_1,\ldots, t_m$, be distinct integers. Then there  exists
an ordinal $\nu<\lambda$ such that $\beta_i+t_i\gamma_s$ are different for all $s>\nu$. Also  
there exists an integer $1\leq r\leq m$ such that
$$\beta_i+t_i\gamma_s>\beta_r+t_r\gamma_s$$ 
for all $i\not = r$ and $s>\nu$.
\end{Lemma}
The following two results use the above lemma.
\begin{Lemma} (\cite[Lemma 3]{P2}) \label{ka}
Let $V \subset V'$ be an  immediate extension of valuation rings, $K\subset K'$ its fraction field extension and  $(v_i)_{i<\lambda}$ an algebraic pseudo convergent sequence in $V$, which has a pseudo limit $x$ in $V'$, but no pseudo limit in $K$. Set\\
 $x_i=(x-v_i)/(v_{i+1}-v_i)$. Let $s\in {\bf N}$ be   the
 minimal degree of the polynomials $f\in V[Y]$ such that
 $\val(f(v_i))<\val(f(v_j))$ for large $i<j<\lambda$ and $g\in V[Y]$ a polynomial with $\deg g<s$. Then there exist $d\in V\setminus \{0\}$ and $u\in V[x_i]$ for some $i<\lambda$ with $g(x)=du$ and $\val(u)=0$.
\end{Lemma}

\begin{Lemma} (\cite[Lemma 4]{P2}) \label{kap}
Let $V \subset V'$ be an  immediate extension of valuation rings, $K\subset K'$ its fraction field extension and  $(v_i)_{i<\lambda}$ an algebraic pseudo convergent sequence in $K$, which has a pseudo limit $x$ in $V'$, algebraic over $V$, but no pseudo limit in $K$. Assume that $h=$Irr$(x,K)$ is from $V[X]$. Then 
\begin{enumerate}
\item $\val(h(v_i))<\val(h(v_j))$ for large $i<j<\lambda$,
\item if $h$ has minimal degree among the polynomials $f\in V[X]$ such that
 $\val(f(v_i))<\val(f(v_j))$ for large $i<j<\lambda$, then $V''=V'\cap K(x)$  is a filtered  union of its complete intersection $V$-subalgebras. 
\end{enumerate}
\end{Lemma}

\begin{Remark}\label{r0}  The extension $V\subset V''$ from (2) of the above lemma is isomorphic with the one constructed in \cite[Theorem 3]{Kap}. 
\end{Remark}

\begin{Lemma}   ( \cite[Lemma 5.2]{KV}) \label{limit} Assume that $ (v_j)_j$ is a pseudo-convergent sequence in $V$ with a pseudo limit $x$ in an extension $V'$ of $V$ but with no pseudo limit in $V$. Let $f\in V[X]$ be a polynomial. Then 
$(f(v_j))_j$ is ultimately pseudo-convergent and $f(x)$ is a pseudo limit of it.
\end{Lemma}

An interesting situation is given below.

\begin{Example}(Anonymous Referee of a previous paper) \label{E} Let $R={\bf F}_2[[t]]$ be the formal power series ring in $t$ over ${\bf F}_2={\bf Z}/2{\bf Z}$. Then the fraction field ${\bf F}_2((t))$ of $R$ has a $t$-adic valuation which has a canonical extension denoted by $\val$ to the perfect hull $K= {\bf F}_2((t))^{1/2^{\infty}}$ of  ${\bf F}_2((t))$. Let $V$ be the valuation ring defined by $\val$ on $K$. Let $a$ be a root of the polynomial $f=X^2-X-(1/t)\in K[X]$. Then the  extension of $\val$ to $K(a)$ gives an immediate extension $K\subset K(a)$ with  a value group $\Gamma$. This extension is closed to an example of \cite{O} (see Example \ref{e}). We can assume that the  root $a$ is a pseudo limit of the pseudo convergent sequence $(a_n)_{n\in {\bf N}}$ given 
by 
$$a_n=\sum_{i=1}^nt^{-2^{-i}},$$
which has no pseudo limit in $K$. Note that $\val(a_n)<0$ and so $a_n\not \in V$ for all $n\in {\bf N}$.

Let $0<\alpha\in \Gamma$ and $b=(b_j)_{j<\lambda}$ be a transcendental pseudo convergent sequence over $K(a)$ such that $\val(b_j)=\alpha$ for high enough $j$. By \cite[Theorem 2]{Kap} there exists an unique immediate transcendental extension $K(a)\subset L=K(a,z)$ such that $z$ is a pseudo limit of $b$ (the valuation of $L$ is still denoted by $\val$).  Then $\val(z)=\alpha>0$. Set $x=z+a\in L$. We have 
$$\val(f(x))=\val((x-a)^2-(x-a)+a^2-a-(1/t))=\val(x-a)+\val(x-a-1)=\alpha+0>0.$$

Note that 
$$\val(x-a_n)=\min\{\val(x-a),\val(a-a_n)\}=\val(a_{n+1}-a_n)$$
because $\val(x-a)=\alpha>0$ and  $\val(a-a_n)=\val(a_{n+1}-a_n)<0$, that is $x$ is also a pseudo limit of $(a_n)_n$. Also we see that  $\val(a_n^2-a_m-(1/t))<0$ for all $n,m\in {\bf N}$ with $m\not =n-1$
and  $a_n^2-a_m-(1/t)=0$ when $m=n-1$. Consequently,  $\val(a_n^2-a_m-(1/t))$ is either $<0$, or it is $\infty$  for all $n,m\in {\bf N}$ and it follows that   
 $\val(a_n^2-a_m-(1/t))$ is never $\val(f(x))$.
 
 We see that the ultimately pseudo convergent sequence $(a_n^2)_{n\in {\bf N}}$  (using Lemma \ref{limit}) has no pseudo limit in $K$ because $K^2=K$ and a pseudo limit of $(a_n^2)$ in $K$ would give a pseudo limit of $(a_n)$ in $K$, which is false. Set $y_0=x$, $y_1=y_0^2$, $g=tY_1-tY_0-1 \in K[Y_0,Y_1]$. We have $g(y_0,y_1)=tf(x)$. Then $\val(g(a_n^2,a_m))$, $n,m\in {\bf N}$ is never $\val( g(y_0,y_1)) $ because  $\val(a_n^2-a_m-(1/t))$ is never $\val(f(x))$ as we have seen. So Lemma \ref{ka} does not work for
  polynomials in two variables.
\end{Example}

\begin{Lemma}(\cite[Lemma 6]{P2}) \label{com} 
Let $B$ be a complete intersection algebra over a ring $A$ and $C$ a complete intersection  algebra over $B$. Then  $C$ is a complete intersection algebra over $A$.
\end{Lemma}

Using Lemma \ref{kap} and the above lemma we  proved the following result.

\begin{Theorem}  (\cite[Theorem 1]{P2})\label{T} Let  $ V'$ be an  immediate extension  of a valuation ring $V$  and  $K\subset K'$ the fraction field extension. If $K'/K$ is algebraic  then $V'$ is a filtered
 union of its complete intersection $V$-subalgebras of finite type.
\end{Theorem}

As a consequence the next theorem follows.

\begin{Theorem}\label{T'} Let $V$ be a  valuation ring containing its residue field $k$   with a free value group (for example
when it is finitely generated) $\Gamma$.  Assume that $y$ is a system of elements of $V$ such that $\val(y)$ is a basis of $\Gamma$ and the fraction field  of $V$ is an algebraic extension of $ k(y)$. 
 Then $V$ is a filtered  union of its complete intersection $k$-subalgebras of finite type.
\end{Theorem}
\begin{proof}
 By \cite[Lemma 27 (1)]{P} we see that $ V$ is an immediate extension of the valuation ring $ V_0=V\cap k(y)$, which is a filtered  union of localizations of its polynomial $ k$-subalgebras (see \cite[Theorem 1, in VI (10.3)]{Bou}, or \cite[Lemma 26 (1)]{P}).  In fact, important here is that there exists a cross-section $s:\Gamma\to K^*$ (a {\em cross-section}
of $V$ is a section
in the category of abelian groups of the valuation map $\val : K^* \to \Gamma$). Since $\Gamma$ is free we define $s$ by $\val(y)\to y$. 
Now it is enough to apply Theorem \ref{T}
 for $V_0\subset V$. 
\hfill \ \end{proof}

The above theorem connects the theory of valuation rings with the theory of complete intersection local rings. We remind that for a given complete intersection local ring $(R,\mm,k)$ a minimal free resolution over $R$ of the residue field $k$ is described in \cite[Theorem 4]{Ta} and others as 
 \cite[Theorem 2.7]{As}, \cite[§6]{Av}, \cite[Proposition 1.5.4]{GL},  \cite[§3]{HM}, \cite[ Theorem 2.5]{AHS}, \cite[Theorem 4.1]{NV}.

 \end{document}